\documentclass[11pt]{amsart}
 \usepackage[dvips]{epsfig}
 \usepackage{comment}
 \usepackage{enumerate}
 \usepackage{amsgen, amstext,amsbsy,amsopn, amsthm, amsfonts,amssymb,amscd,amsmat
 h, euscript,enumerate,url,verbatim,calc,xypic, mathtools}
 \usepackage{mathabx}
\usepackage{mathtools}
\DeclareMathOperator{\Hessian}{Hess}
 \usepackage{latexsym}
 \usepackage{graphics}
 \usepackage{color}
\usepackage{mathabx}
\newcommand{\proset}{\,\mathrel{\lower 4pt\hbox{$\scriptscriptstyle/$}
\mkern -14mu\subseteq }\,} 

 \newtheorem{theorem}{Theorem}[section]
  \newtheorem{corollary}[theorem]{Corollary}
 
 \newtheorem{proposition}[theorem]{Proposition}

\newtheorem{remark}[theorem]{Remark}

 \newtheorem{example}[theorem]{Example}

\numberwithin{equation}{section}

\usepackage{amsmath}
\usepackage{tikz-cd}
\usepackage[ colorlinks=true, linkcolor=blue, citecolor=blue, urlcolor=blue] {hyperref}
\usepackage[backend=biber, style=alphabetic, sorting=ynt]{biblatex}
\bibliography{mybib}
\title{A STRUCTURAL ANALYSIS OF WARPED PRODUCT YAMABE GRADIENT SOLITONS}
\author{JAHNABI CHAKRABORTI, ANANDATEERTHA MANGASULI}
\begin{document}
\begin{abstract}
    We investigate Yamabe gradient solitons, which are warped product manifolds. We show that the fiber of a nontrivial warped product Yamabe gradient soliton has constant scalar curvature. Based on this result, we obtain a specific class of warped products that cannot be nontrivial solitons. Furthermore, we derive estimates of the scalar curvature and the soliton function of warped product solitons.
\end{abstract}
\maketitle

\section{\textbf{Introduction}}
In 1960, H. Yamabe \cite{Yamabe} posed the Yamabe problem, which involves finding a Riemannian metric of constant scalar curvature on a smooth manifold by a conformal change of a given metric. During the 1980s, R. Hamilton \cite{HA1} proposed a parabolic differential equation approach to address this problem. His method consists of evolving an initial metric through a geometric flow, called the Yamabe flow, to obtain a constant scalar curvature metric. The Yamabe flow preserves the conformal structure of the metric. Solving the Yamabe problem through the Yamabe flow has been the focus of extensive research in recent years; see the survey by Brendle \cite{brendle} and the references cited therein. Yamabe solitons are self-similar solutions to the Yamabe flow. Every compact Yamabe soliton has constant scalar curvature \cite{di}. There are noncompact Yamabe solitons that do not have constant scalar curvature, such as the Cigar soliton. In this paper, we focus on a subclass of Yamabe solitons, referred to as noncompact Yamabe gradient solitons. Though Yamabe gradient solitons are solutions to the Yamabe flow, there is an equivalent algebraic equation point of view definition.

 A Riemannian metric $g$ on a connected smooth $n-$manifold $M^n$ is called a \textit{Yamabe gradient soliton} if there exist a smooth function $h: M\rightarrow \mathbb{R}$ and a constant $\rho \in \mathbb{R}$ such that 
        \begin{equation} \label{GYS}
            \Hessian h=(R-\rho) g\hspace{3mm} \text{on} \hspace{1mm} M
        \end{equation}
    where $R$ is the scalar curvature w.r.t the metric $g$. It is denoted by $(M^n, g, h, \rho)$. A Yamabe gradient soliton $(M^n, g, h, \rho)$ is said to be steady, shrinking, and expanding if $\rho=0$, $\rho>0$, and $\rho<0$, respectively. The function $h$ is called the soliton function of the Yamabe gradient soliton. If $h$ is a constant function on $M$, then it is a trivial Yamabe gradient soliton.   If $\rho$ is a smooth function on $M$, then $(M^n, g, h, \rho)$ is called an\textit{ almost Yamabe gradient soliton}. It is known that every compact Yamabe gradient soliton has constant scalar curvature; hence, trivial \cite{di} \cite{hsu}. In 2012, Cao, Sun, and Zhang \cite{CSY} showed that every complete nontrivial n-dimensional Yamabe gradient soliton admits a warped product structure with a 1-dimensional base and an (n-1)-dimensional fiber of constant scalar curvature. In 2017, Tokura, Adriano, and Pina \cite{TBAP2} started working on the Yamabe gradient soliton with warped product structure, see \S~$2$, where the fiber has constant scalar curvature and the soliton function depends only on the base. In 2023, the authors proved the following:
\begin{theorem}\cite{TBAP1} \label{thm1}
    Let $(M^{n+m}=B^n\times_f F^m, g= g_{B}+f^2 g_{F}, h, \rho)$ be a complete and nontrivial warped product Yamabe gradient soliton. Then the soliton function $h$ does not depend on $F^m$ if and only if the fiber manifold $F^m$ has constant scalar curvature.
\end{theorem}

In this paper, motivated by Theorem \ref{thm1}, we study the structure of nontrivial warped product Yamabe gradient solitons, without assuming completeness of the metric. As examples, we refer to punctured Euclidean space, Example \ref{exm4}, and punctured Cigar soliton, Example \ref{exm5}. This investigation lays the groundwork for understanding noncomplete Yamabe gradient solitons. It turns out that dropping the assumptions of completeness of the warped metric and constant scalar curvature fiber, the soliton function of every nontrivial warped product Yamabe gradient soliton depends only on the base. As a result, the fiber has constant scalar curvature, see Theorem \ref{thm2} in \S~ $2$. Based on this special structure, we establish a certain class of warped product manifolds with noncompact bases which cannot be nontrivial solitons, see Theorem \ref{thm3} in \S~$2$. For nontrivial solitons, critical points of the soliton function can be utilized as a tool to infer several geometric properties, including curvature. 
In Corollary \ref{cor3} in \S~$3$, we show that the soliton function has no critical point if the base is complete. Considering the information about the nonexistence of critical points of the soliton function, we conclude this paper by estimating a lower bound of the scalar curvature of nontrivial warped product solitons, see Theorems \ref{thm5} and \ref{thm6} in \S~$4$.

\section{\textbf{Triviality of Warperd Product Yamabe gradient Solitons} }

The \textit{warped product} $M=B\times_f F$ is the manifold $B\times F$ equipped with the Riemannian structure $g$ such that 
\small{\begin{center}
   $g(X, Y)=g_B(\pi_*(X), \pi_*(Y))+(f\circ\pi)^2 g_F(\eta_*(X), \eta_*(Y))$ 
\end{center}}
for any tangent vectors $X, Y \in TM$, where $f:M \rightarrow B$ is a nonconstant positive smooth function, $\pi: M \rightarrow B$ and $\eta: M\rightarrow F$ are the natural projection maps. 
The function $f$ is called the warping function of the warped product. 
If $X\in T_pB$, $p\in B$ and $q\in F$, then the lift $\Bar{X}$ of $X$ to $(p,q)$ is the unique vector in $T_{(p,q)}M$ such that $\pi_*(\Bar{X})=X$. For a vector field $X\in \mathcal{X}(B)$, the lift of $X$ to $M$ is the vector field $\Bar{X}$ whose value at each $(p,q)$ is the lift of $X_p$ to $(p,q)$. The set of all horizontal lifts is denoted by $\mathcal{L}(B)$. Similarly, $\mathcal{L}(F)$ is the set of all vertical lifts.\\
Now, we have the following theorem:
\begin{theorem}\label{thm2}
    If $(M^{n+m}=B^n\times_f F^m, g= g_{B}+f^2 g_{F}, h, \rho)$ is a nontrivial  warped product Yamabe gradient soliton, then the following hold:
    \begin{enumerate}
       \item The soliton function $h$ depends only on the base $B^n$;\label{res1}
       \item The fiber $(F^m, g_F)$ has constant scalar curvature.
    \end{enumerate}
\end{theorem}
\begin{proof} 
    \begin{enumerate}
    \item Let, $\Bar{X}=(X, 0)\in \mathcal{L}(B)$ and $\Bar{V}=(0, V)\in \mathcal{L}(F)$ be a horizontal and vertical lift of vector fields $X\in \mathcal{X}(B)$ and $V\in \mathcal{X}(F)$. Due to \cite[Proposition $3.1. (2)$]{chen} we can write,
    \begin{equation}\label{FE1}
        \nabla^2 h(\Bar{X}, \Bar{V})=
        \Bar{X}\Bar{V}(h)-(\Bar{X}ln f)g(\nabla h, \Bar{V})
    \end{equation}
    On the other hand,
    \begin{equation}\label{HESM}
        \nabla^2 h(\Bar{X}, \Bar{V})=g(\nabla_{\Bar{X}} {\nabla h}, \Bar{V})
    \end{equation}
    The orthogonal decomposition of $\nabla h$ is
        $\nabla h= \mathcal{H}(\nabla h)+\mathcal{V}(\nabla h)$, 
    where $\mathcal{H}(\nabla h)$ and $\mathcal{V}(\nabla h)$ are the horizontal and vertical components of $\nabla h$.
    Using \eqref{GYS} and \cite[Proposition $3.1.$]{chen} in \eqref{HESM} we get,
    \begin{equation}\label{MP}
       \Bar{X} (ln f)g({\mathcal{V}(\nabla h)}, \Bar{V})=0
    \end{equation}
    Which implies that at each point $(p,q)\in B\times F=M$,
   \begin{equation}\label{FE3}
        \Bar{X} f=0 \hspace{5mm}or\hspace{5mm} g({\mathcal{V}(\nabla h)}, \Bar{V})=0.
   \end{equation}
    Using 
   \eqref{HESM} and \eqref{MP} in \eqref{FE1} we get,
       $\Bar{X}\Bar{V}(h)=0$.
   Therefore, $h$ can be of the form $h=h_1+h_2$, for some smooth functions $h_1\in C^{\infty}(B)$ and $h_2\in C^{\infty}(F)$.\\
   Since, $f$ is nonconstant function then there exist a point $p_0\in B$ and  $\Bar{E}_i=(E_i, 0)\in \mathcal{L}(B)$, where $E_i$ is a basis vector field of $TB$ such that $\Bar{E}_i(f)(p_0,q)\neq 0$ for all $q\in F$.\\
   Taking $\Bar{X}=\Bar{E_i}$ in \eqref{FE3} and evaluating at $(p_0, q)$ we get,
      ${|\nabla^F h_2|^2}(q)=0 \hspace{3mm} \text{for all $q\in F$}$

   \item Let, $\Bar{V}=(0,V)\in \mathcal{L}(F)$ and $\Bar{W}=(0, W)\in\mathcal{L}(F)$ be vertical lifts of vector fields of vector fields $V, W\in \mathcal{X}(F)$. Using the result (\ref{res1}) and \cite[Proposition $3.1. (3)$]{chen} we get,
  
     \begin{center}\label{FE4}
       $\nabla^2 h(\Bar{V}, \Bar{W})=\frac{g(\Bar{V},\Bar{W})}{f}\nabla f(h)$
  \end{center}
    \begin{equation}\label{SCE}
     \implies (R-\rho)= \frac{g(\nabla h, \nabla f)}{f}
  \end{equation}
  \end{enumerate}
 Hence by \cite[Proposition $3.3$]{chen} 
  the scalar curvature of the fiber $(F, g_F)$ is constant.

\end{proof}

       





\begin{example}\label{exm4}\cite{chen}
   Let $\mathbb{R}_{+}=\{r\in\mathbb{R}: r>0\}$. The punctured Euclidean $n-$ space $\mathbb{E}^n_*$ is the warped product  $\mathbb{R}_{+}\times_r \mathbb{S}^{n-1}$ equipped with the warped product metric $g_{\mathcal{E}}=dr^2+r^2 g_{can}$,where $g_{can}$ is the round metric of the sphere $\mathbb{S}^{n-1}$.
   $(\mathbb{E}^n_*,g_{\mathcal{E}})$ is a noncomplete nontrvial warped product Yamabe gradient soliton 
   with soliton function 
        $h(r,\theta)=-\rho \frac{r^2}{2}$
    where $\rho$ is a nonzero constant.
\end{example}

As an immediate consequence of Theorem \ref{thm2}, we have an existential criterion of a Yamabe gradient soliton on a warped product manifold.  
\begin{corollary}\label{cor1}
    A warped product manifold $(M^{n+m}=B^n\times_f F^m, g= g_{B}+f^2 g_{F})$  is a nontrivial Yamabe gradient soliton 
    if and only if the fiber $(F^m, g_F)$ has constant scalar curvature and  there exists a smooth function $h_1\in C^{\infty}(B)$ along with a constant $\rho\in \mathbb{R}$ such that $(B^n, g_B, h_1, \lambda)$ is a nontrivial almost Yamabe gradient soliton with 
    \begin{equation}\label{FE10}
        \lambda=-\frac{R_F}{f^2}+2m\frac{\Delta_B f}{f}+m(m-1)\frac{|\nabla^B f|^2}{f^2}+\rho
    \end{equation}
    and
    \begin{equation}\label{FE14}
        \frac{g_B(\nabla^B f, \nabla^B h_1)}{f}=R_B-\lambda=R-\rho
    \end{equation}
   where $R_B$ and $R$ are the scalar curvatures of $B$ and $M$ respectively.
\end{corollary}

By the technique used in \cite[Theorem~$1.3.$]{TBAP2}, the topology of the base can be proved.
\begin{corollary}\label{cor2}
   If $(M^{n+m}=B^n\times_f F^m, g=g_{B}+f^2 g_{F}, h, \rho)$  is a nontrivial warped product Yamabe gradient soliton 
   , then the base $B$ is noncompact.
\end{corollary}
\begin{remark}\label{rmk1}
    In \cite{Grajales}, the authors obtained a similar result, Theorem \ref{thm2}, through a different approach, wherein the completeness assumption has been taken in the definition of Yamabe gradient solitons. However, it seems that their approach cannot be extended to noncomplete settings, such as the punctured Euclidean space considered in Example \ref{exm4}.
\end{remark}
The converse of the above Corollary \ref{cor2} is not true in general, i.e., Examples \ref{exm1} and \ref{exm3}.
\begin{example}\cite{TBAP1}\label{exm1}
    Let ${\mathbb{R}_{+}^3}=\{(x_1, x_2, x_3)\in \mathbb{R}^3:x_3>0\}$. The warped product $M^6={\mathbb{R}_{+}^3}\times_f \mathbb{R}^3$ equipped with the warped product metric $g=\frac{1}{{x_3}^2}(dx_1^2+dx_2^2+dx_3^2)+f^2(dy_1^2+dy_2^2+dy_3^2)$, where the warping function $f(x_1, x_2, x_3)=\frac{1}{x_3}$, is a complete noncompact trivial Yamabe gradient soliton with $^{M}\mathrm{Ric}=-\frac{5}{6}g$. 
\end{example}
\begin{example}\cite{chen}\label{exm3}
    The warped product $M^{n}=(0, \frac{\pi}{2})\times_{\sin{r}} \mathbb{S}^{n-1}$ equipped with the warped product metric $g_1=dr^2+{\sin^{2}{r}} g_{can}$, where $g_{can}$ is the round metric of the sphere $\mathbb{S}^{n-1}$, 
    is a noncomplete trivial Yamabe gradient soliton with sectional curvature $1$. 
\end{example}
We find a property of the warping function $f$ such that the warped product becomes trivial Yamabe gradient solitons.
\begin{theorem}\label{thm3}
    Let $(M^{n+m}=B^n\times_f F^m, g=g_{B}+f^2 g_{F}, h, \rho)$ be a warped product  Yamabe gradient soliton 
    with noncompact base $B$. If $ \Hessian f=cfg_{B}$, for some nonzero constant $c$, then $(M^{n+m}, g)$ is trivial Yamabe gradient soliton. Moreover, the warped product $(M^{n+m}, g)$ has constant scalar curvature.
\end{theorem}
\begin{proof}
    Suppose $(M=B^n\times_f F^m, g=g_B+f^2g_F, h, \rho)$ is a nontrivial Yamabe gradient soliton. Then Corollary \ref{cor1} implies that $(B, g_B, h_1, \lambda)$ is a nontrivial almost Yamabe gradient soliton with $\lambda$ defined as \eqref{FE10} and $h=h_1\circ \pi$.\\
    Let, $\Bar{X}=(X,0)\in \mathcal{L}(F)$ and $\Bar{Y}=(Y, 0)\in \mathcal{L}(B)$ be horizontal lifts of two vector fields $X, Y\in \mathcal{X}(B)$.
    From the definition of almost Yamabe gradient soliton $(B, g, h_1, \lambda)$ and \eqref{FE14} we get,
   \begin{equation}\label{FE25}
       \nabla_{\Bar{X}}\nabla h=(R-\rho) \Bar{X},
   \end{equation}
 where $R$ is the scalar curvature of $(M, g)$.\\
 By \eqref{FE25} we have,
   \begin{align*}
        R(\Bar{X}, \nabla h) \Bar{Y}&=\Bar{Y}(R) \Bar{X}-g(\Bar{X}, \Bar{Y})\nabla R
    \end{align*}
     Choosing $\Bar{X}=\Bar{Y}=\nabla h$ we get,
     \begin{equation}\label{FE5}
         \nabla h(R) \nabla h= |\nabla h|^2 \nabla R
     \end{equation}
    Since $R-\rho=\frac{g(\nabla f, \nabla h)}{f}$ and $\Hessian f(\nabla h, \nabla h)=cf|\nabla h|^2$, for some nonzero constant $c$, then
    \begin{equation}\label{FE6}
        \nabla h(R)=c |\nabla h|^2
    \end{equation}
Using \eqref{FE6} in \eqref{FE5} we get, at each point $p\in M$,
\begin{eqnarray*}
    \text{either} \hspace{3mm} |\nabla h|=0 \hspace{3mm} \text{or} \hspace{3mm} \nabla R=c\nabla h.
\end{eqnarray*}
Define, $G=\{p\in M: ||\nabla h|\neq 0\}$. 

On $G$, the elliptic equation of the scalar curvature of a Yamabe gradient soliton in \cite[Page 363]{daskalopoulos} gives,
\begin{center}
    $(n+m-1) (n+m)c (R-\rho)+\frac{1}{2}g(\nabla R, \nabla h)+R(R-\rho)=0$
\end{center}
Differentiating w.r.t $\nabla h$ and using \eqref{FE6} we have,
\begin{equation}\label{FE7}
    (n+m-1)(n+m)c^2|\nabla h|^2+\frac{1}{2} \nabla h (\nabla h(R))+ c(R-\rho) |\nabla h|^2+cR|\nabla h|^2=0
\end{equation}
Again by \eqref{FE6} and 
\cite[Equation $(2.1)$]{CSY} we get, $\nabla h (\nabla h(R))=c\nabla h(|\nabla h|^2)=2c (R-\rho)|\nabla h|^2$.\\
Therefore \eqref{FE7} implies,
\begin{center}
    $c|\nabla h|^2\{(n+m-1)(n+m)c+(3R-2\rho)\}=0$
\end{center}
 Hence on $G$, 
     $R|_G= $constant. 
 which is a contradiction.
\end{proof}

\section{\textbf{Critical points of the soliton function}}
Due to Corollary \ref{cor1}, the geometric structure of a nontrivial warped product soliton depends on the geometric structure of the base, which in turn is a nontrivial almost Yamabe gradient soliton. In 1965, Tashiro \cite{tashiro} established the conformal structure of complete nontrivial gradient solitons through the number of critical points of the soliton function. In contrast, we obtain the isometric structure of complete nontrivial almost Yamabe gradient solitons. As a consequence, we prove that the soliton function of a warped product soliton with a complete base does not have any critical point.    
\begin{proposition}\label{prop3}
    Let $(B^n, g_B, h, \lambda)$ be a compact almost Yamabe gradient soliton. Then $(B^n, g_B)$ is trivial if one of the following holds:
    \begin{enumerate}
        \item $\int_B {^{B}\mathrm{Ric} (\nabla^B h, \nabla^B h)} \le 0$
        \item $\int_B g_B(\nabla^B \lambda, \nabla^B h)\le 0$. (This is the corrected sign of the integral \cite[Theorem $1.3 (2)$]{barbosa})
      
    \end{enumerate}
\end{proposition}

\begin{proof}
    Applying Green's theorem and Lemma 
    $2.3. (2)$ of \cite{barbosa} we get,
          \begin{eqnarray*}
             \int_B (R_B-\lambda)^2 dV=\frac{1}{n(n-1)}\int_B {^{B}\mathrm{Ric}(\nabla^B h, \nabla^B h)} dV
          \end{eqnarray*}
    Again we apply Theorem II.9 of \cite{bourguignon} to deduce that,
           \begin{eqnarray*}
           \int_B (R_B-\lambda)^2 dV=\frac{1}{n}\int_B g_B(\nabla^B \lambda, \nabla^B h) dV.
           \end{eqnarray*}
\end{proof}
\begin{theorem}\label{thm4}
    Let $(B^n, g, h, \rho)$ be a complete nontrivial almost Yamabe gradient soliton. Then, $h$ has atmost two critical points 
   and 
    \begin{enumerate}
       \item If $h$ has two critical points $p_1, p_2$; then $B^n$ is compact. $(B^n\setminus \{p_1, p_2\}, g)$ isometric to $((0, \pi), dr^2)\times_{|\nabla^B h|} (\mathbb{S}^{n-1}, g_{\mathbb{S}^{n-1}})$, where $g_{\mathbb{S}^{n-1}}$ is a metric on the sphere $\mathbb{S}^{n-1}$ and there exists a point $p\in B^n\setminus \{p_1, p_2\}$ such that $Hess(|\nabla^B h|)(\nabla^B h, \nabla^B h)(p)<0$.
        \item If $h$ has a critical point, then $(B^n, g)$ is isometric to $([0, \infty), dr^2)\times_{|\nabla^B h|} (\mathbb{S}^{n-1}, g_{can})$, where $g_{can}$ is the round metric on the sphere $\mathbb{S}^{n-1}$.
        \item If $h$ has no critical point, then  $(B^n, g)$ is isometric to $(\mathbb{R}, dr^2)\times_{|\nabla^B h|}(N^{n-1}, g_N)$, where $(N^{n-1}, g_N)$ is a complete Riemannian manifold.
    \end{enumerate}
\end{theorem}
\begin{proof} 
Tashiro \cite{tashiro} showed that the soliton function \( h \) has at most two critical points. By choosing a regular level set \( N = h^{-1}(c) \) of $h$, he proved the following:

\begin{enumerate}
    \item If \( h \) has one or two critical points, then \( I \times N \) has a diffeomorphism into \( M \).
    \item If \( h \) has no critical points, then \( I \times N \) is diffeomorphic to \( M \).
\end{enumerate}

Where the interval \( I \) is of the form \( (r_1, r_2) \), \( (r_1, \infty) \), or \( (-\infty, \infty) \), depending on whether \( h \) has two, one, or zero critical points, respectively.
Hence, using the coordinates of $I\times N$, the metric $g$ of $M$ can be written in the following form, except at critical points of $h$.
\begin{center}
    $g=dr^2+g_{ij} (r, \theta) d{\theta}^i d{\theta}^j$, 
\end{center}
where $r\in I$ and $\theta=({\theta}^2,...,{\theta}^n)$ is any local coordiantes on $N$.\\
On $M$ except at critical points of $h$, we can write,
\begin{center}
    $\nabla h=h'(r) \frac{\partial}{\partial r}$
\hspace{2mm}
and
\hspace{2mm}
    $\nabla^2 h=\nabla(dh)=\nabla (h'(r) dr)=h''(r) dr\otimes dr+h'(r) \nabla^2 r$.
\end{center}
Since $\Gamma_{rr}^r=0$, $\Gamma_{ir}^r=0$ and $\Gamma_{ij}^r=-\frac{1}{2} \frac{\partial}{\partial r} g_{ij}$ we get,
\begin{align}\label{HE}
     h''(r)&= (R_B-\lambda) \hspace{2mm}
\text{and}
\hspace{2mm} \frac{h'(r)}{2} \frac{\partial}{\partial r}g_{ij}= (R_B-\lambda) g_{ij}=h''(r) g_{ij}
\end{align}
We may assume $h'(r)>0$ on $M\backslash\{\text{critical points of $h$}\}$. From \eqref{HE} it follows that,
\begin{equation}\label{HE3}
    g_{ij}(r, \theta)={(\frac{h'(r)}{h'(r_0)})}^2 g_{ij}(r_0,\theta).
\end{equation}
Here the level surface $\{r=r_0\}$ corresponds to $N$. 
\\
\textbf{Case I:} If \( h \) has a critical point \( p_0 \in B \), then, according to Tashiro's construction, \( r \) is the distance function from the critical point \( p_0 \). So, the level surfaces of $h$ are geodesic spheres centered at $p_0$ which are diffeomorphic to $(n-1)$-dimensional sphere $\mathbb{S}^{(n-1)}$. By the smoothness of  $g$ at $p_0$, the metric $\Bar{g}$ on $N$ is round metric by \cite[Lemma 9.114.]{besse}.\\
\textbf{Case II:} If \( h \) admits two critical points \( p_0, p_1 \in M \), then, as in analogous case, the level sets of \( h \) are geodesic spheres centered at \( p_0 \), each diffeomorphic to the \((n-1)\)-dimensional sphere \( \mathbb{S}^{n-1} \). Due to Proposition \ref{prop3} and \cite[Proposition $3.3$ $(1)$]{chen}, we have our required isometric structure.
\end{proof}

The following corollary will be applied to estimate the scalar curvature of nontrivial warped product solitons.

\begin{corollary}\label{cor3}
    If $(M^{n+m}=B^n\times_f F^m, g= g_{B}+f^2 g_{F}, h, \rho)$ is a nontrivial warped product Yamabe gradient soliton with complete base $(B^n, g_B)$, then 
    $h$ has no critical point on $M$ and $(M^{n+m}, g)$ is isometric to a doubly warped product
        $(\mathbb{R}\times_{|\nabla h|}N^{n-1}\times_f F^m, dr^2+|\nabla h|^2 g_N+f^2 g_F)$ with $(N^{n-1}, g_N)$ is a complete Riemannian manifold. 

\end{corollary}

\begin{corollary}\label{cor5}
    If $(M^{n+m}=B^n\times_f F^m, g= g_{B}+f^2 g_{F}, h, \rho)$ is a nontrivial warped product Yamabe gradient soliton with complete base $(B^n, g_B)$ and $g(\nabla log f, \nabla h)=c$, where $c\in \mathbb{R}$, then $c=0$ and  $(M^{n+m}, g)$ is isometric to 
        $(\mathbb{R}\times N^{n-1}\times_f F^m, dr^2+ g_N+f^2 g_F)$ with $(N^{n-1}, g_N)$ is a complete Riemannian manifold.
\end{corollary}
\begin{remark}
    The corollary \ref{cor5} is a generalization of \cite[Proposition $1.5$]{TBAP2}, where an isometric structure was obtained under the assumptions that the soliton function depends on the base, the fiber is of constant scalar curvature, and $g(\nabla log f, \nabla h)=c\neq 0$.
\end{remark}

\section{\textbf{Scalar Curvature and Potential Function Estimation}}
 Now, let us discuss the scalar curvature of warped product Yamabe gradient solitons. Trivial solitons always have constant scalar curvature $R=\rho$. There are nontrivial warped product Yamabe gradient solitons of scalar curvature being unbounded below.
\begin{example}\cite{TBAP1}\label{exm6}
    Let $\mathbb{R}^3_*=\{(x_1, x_2, x_3)\in \mathbb{R}^3:x_1+x_2+x_3>0\}$. The warped product $M^6={\mathbb{R}^3}_*\times_f \mathbb{R}^3$ equipped with the warped metric $g=\frac{20}{x_1+x_2+x_3}(dx_1^2+dx_2^2+dx_3^2)+f^2 (dy_1^2+dy_2^2+dy_3^2)$, where the warping function $f(x_1, x_2, x_3)=\sqrt{\frac{20}{x_1+x_2+x_3}}$, is a  steady nontrivial warped product Yamabe gradient soliton with the soliton function 
        $h(x_1, x_2, x_3, y_1, y_2, y_3)=20 ln (x_1+x_2+x_3)$. 
    The scalar curvature of $(M^6, g)$ is $R=-\frac{1}{x_1+x_2+x_3}$, which is unbounded below.
\end{example}
In this section, we give a criterion by which the scalar curvature is bounded below and estimate the soliton function. First, we 
Consider the base to be complete. 
\begin{theorem}\label{thm5}
        Let $(M^{n+m}=B^n\times_f F^m, g=g_B+f^2g_F, h, \rho)$ be a nontrivial warped product Yamabe gradient soliton with complete base $(B^n, g_B)$. If 
        \begin{eqnarray*}
        \inf_{|\nabla^B h|\neq0}\frac{(n+m-1)}{n-1}{^{B}}\mathrm{Ric}(\frac{\nabla^B h}{|\nabla^B h|}, \frac{\nabla^B h)}{|\nabla^B h|})-\frac{1}{2(n+m-1)}{\Hessian} h(\frac{\nabla^B h}{|\nabla^B h|}, \frac{\nabla^B h}{|\nabla^B h|})=A,
        \end{eqnarray*}
          for some $A\in \mathbb{R}$,
          then the scalar curvature $R$ of $(M^{n+m}, g)$ is bounded from below and there exists a positive constant $C$ such that the following holds:
        \begin{enumerate}
              \item If $\rho\ge0$ then $R\ge -(n+m-1)C|A|$;
            \item  If $\rho< 0$ then $R\ge -(n+m-1)C|A|+\rho$.
        \end{enumerate}
        
\end{theorem}

\begin{proof}
   Corollary \ref{cor1} and Corollary \ref{cor3} together imply that $(B^n, g_B, h_1, \lambda)$ is a complete nontrivial almost Yamabe gradient soliton isometric to $(\mathbb{R}, dr^2)\times_{|\nabla^B h_1|}(N, \Bar{g}_N)$, where $(N, \Bar{g}_N)$ is a complete Riemannian manifold, $h=\pi^*h_1$ and $\lambda$ defined as \eqref{FE10}. For simplification, we take $h_1=h$.. \\
   Define a function $d:{B}^n\rightarrow \mathbb{R}$ by
   \[
    d(x)=d(r, \theta)=
     \begin{cases} 
      r, & \text{if } r>0, \\
      0, & \text{if }  r=0, \\
      -r, & \text{if }  r<0.
    \end{cases}
   \]
   Consider a tubular neighbouhood of $N$ in $B$ as $(-r_0,r_0)\times N$ for some $r_0>0$.
   For any fixed constant $D>2$, consider $v(x)=\alpha(\frac{d(x)}{Dr_0})R(x)$, where $R$ is the scalar curvature of $(M^{m+n}, g)$ and $\alpha$ is a  smooth nonnegative decreasing function such that 
   \[
    \alpha \equiv
     \begin{cases} 
      1, & \text{if } \frac{d(x)}{Dr_0} \leq \frac{1}{2}, \\
    0, & \text{if }  \frac{d(x)}{Dr_0} \ge 1.
    \end{cases}
   \]
   If $\min_{x\in B} v\ge 0$, then $R\ge 0$ on $(-\frac{Dr_0}{2}, \frac{Dr_0}{2})\times N$.\\
   If $\min_{x\in B} v< 0$, then there exists a point $x_1\in (-Dr_0, Dr_0)\times N$ such that $v(x_1)=\alpha R(x_1)<0$. Since $x_1$ is a minimum of the function $v(x)$ then $\alpha^{'} R(x_1)>0$, $\nabla^B v(x_1)=0$, and $\Delta_B v(x_1)\ge 0$.\\
   \textbf{Case I:} When $x_1\in (-r_0, r_0)\times N$.
   Using  the elliptic equation of the scalar curvature of a Yamabe gradient soliton in \cite[Page 363]{daskalopoulos} 
   we get for shrinking solitons $x_1\not\in (-r_0, r_0)\times N$ and for expanding solitons $R\ge \rho$ on $(-\frac{Dr_0}{2}, \frac{Dr_0}{2})\times N$.\\
  \textbf{Case II:} When $x_1\not\in (-r_0, r_0)\times N$. Let $x_1=(r_1, \theta_1)$, where $\theta_1=(\theta_1^2, ..., \theta_1^n)$.
    Again, by the elliptic equation of the scalar curvature of a Yamabe gradient soliton in \cite[Page 363]{daskalopoulos}, we deduce that
\begin{equation}\label{FE19}
      \Delta_B v=R \Delta_B \alpha+\alpha \Delta_B R+2g_B(\nabla^B \alpha, \nabla^B R)
\end{equation}
    together with,
   \begin{equation}\label{FE20}
     (n+m-1)\Delta_B R+\frac{(n+m-1)m}{f}g_B(\nabla^B f, \nabla^B R)+\frac{1}{2} g_B(\nabla^B R, \nabla^B h)+ R(R-\rho)=0
   \end{equation}
Subsituting $\Delta_B R$ from $\eqref{FE20}$ into $\eqref{FE19}$ we get,
   \small{\begin{eqnarray*}
       \Delta_B v(x_1)=\frac{v(x_1)}{\alpha}[\{\frac{\alpha''}{(Dr_0)^2}-2\frac{{\alpha'}^2}{\alpha}\frac{1}{(Dr_0)^2}\}+\frac{\alpha'}{Dr_0}\{\Delta_B d+g_B(\nabla^B d, \nabla^Blnf^m+\frac{1}{2(n+m-1)}\nabla^B h)\}\\-\frac{v(x_1)}{(n+m-1)}+\frac{\rho}{(n+m-1)}\alpha]
   \end{eqnarray*}}
    For $r_1>0$, 
    since $\frac{\partial}{\partial r}=\phi(r)\nabla^B h$,
    by using \eqref{GYS}, \eqref{FE14}, and 
    \cite[Lemma $2.3. (2)$]{barbosa} we get,
    \begin{eqnarray*}
         \{\Delta_B d+g_B(\nabla^B d, \nabla^Blnf^m+\frac{1}{2(n+m-1)}\nabla^B h)\}(x_1)\le \int_{0}^{r_1} \phi(r)^2 \{(-1-\frac{m}{n-1})
         {^B}Ric (\nabla^B h, \nabla^B h)\\+\frac{1}{2(n+m-1)} {\nabla^2} h(\nabla^B h,\nabla^B h)\}dr+(n-1)\frac{h''(0)}{h'(0)}+m\frac{\partial}{\partial r} lnf(0,\theta_1)+\frac{|\nabla^B h|_N}{2(n+m-1)}
    \end{eqnarray*}
    Since 
    $\inf_B {\frac{(n+m-1)}{n-1}{^{B}}Ric(\frac{\nabla^B h}{|\nabla^B h|}, \frac{\nabla^B h}{|\nabla^B h|})-\frac{1}{2(n+m-1)}{^{B}}{Hess}(h)(\frac{\nabla^B h}{|\nabla^B h|}, \frac{\nabla^B h}{|\nabla^B h|})}=A$,
    \begin{eqnarray*}\label{FE21}
         \{\Delta_B d+g_B(\nabla^B d, \nabla^Blnf^m+\frac{1}{2(n+m-1)}\nabla^B h)\}(x_1)\le -Ad(x_1)+(n-1)\frac{h''(0)}{h'(0)}+m \frac{\partial}{\partial r} lnf (0, \theta_1)\\+\frac{|\nabla^B h|_N}{2(n+m-1)}    
    \end{eqnarray*}
Similarly for $r_1<0$,
   \small{\begin{eqnarray*}\label{FE22}
        \{\Delta_B d+g_B(\nabla^B d, \nabla^Blnf^m+\frac{1}{2(n+m-1)}\nabla^B h)\}(x_1)\le -Ad(x_1)-(n-1)\frac{h''(0)}{h'(0)}
        -m \frac{\partial}{\partial r} lnf (0, \theta_1)\\-\frac{|\nabla^B h|_N}{2(n+m-1)} 
    \end{eqnarray*}}
    Hence,
    \begin{align*}
        \Delta_B v(x_1)\le\frac{v(x_1)}{\alpha}[\{\frac{\alpha''}{(Dr_0)^2}-2\frac{\alpha'}{\alpha}\frac{1}{(Dr_0)^2}\}+\frac{\alpha'}{Dr_0}\{-A d(x_1)+(n-1)\frac{|h''(0)|}{|\nabla^B h|_N}+m |\frac{\partial}{\partial r} lnf (0,\theta_1)|\\
        +\frac{|\nabla^B h|_N}{2(n+m-1)}\}
       -\frac{v(x_1)}{(n+m-1)}+\frac{\rho}{(n+m-1)}\alpha\\
        \le\frac{|v(x_1)|}{\alpha}[\{\frac{|\alpha''|}{(Dr_0)^2}+2\frac{|\alpha'|}{\alpha}\frac{1}{(Dr_0)^2}\}+\frac{|\alpha'|}{Dr_0}\{|A| d(x_1)+(n-1)\frac{|h''(0)|}{|\nabla^B h|_N}+m |\frac{\partial}{\partial r} lnf (0,\theta_1)|\\
        +\frac{|\nabla^B h|_N}{2(n+m-1)}\}
       -\frac{|v(x_1)|}{(n+m-1)}-\frac{\rho}{(n+m-1)}\alpha]
    \end{align*}
     There exist a positive constant $C$ such that $|\alpha'|\le C$, $|\alpha''|\le C$, and l$\frac{{\alpha'}^2}{\alpha}\le C$.\\
    Let $C_1=2m\frac{\partial}{\partial r}\ln f (0,\theta_1)$, $C_2=\frac{|\nabla^B h|_N}{2(n+m-1)}$, and $C_3=(n-1)\frac{|h''(0)|}{|\nabla^B h|_N}$.
    Therefore on $(-\frac{Dr_0}{2},\frac{Dr_0}{2})\times N$,\\
    for $\rho\ge 0$
    \begin{equation}\label{FE23}
        R\ge -(n+m-1)C\{\frac{3}{(Dr_0)^2}+\frac{(C_1+C_2+C_3)}{Dr_0}+|A|\}
    \end{equation}
    and for $\rho<0$
    \begin{equation}\label{FE24}
        R\ge -(n+m-1)C\{\frac{3}{(Dr_0)^2}+\frac{(C_1+C_2+C_3)}{Dr_0}+|A|\}+\rho
    \end{equation}
    Comparing the lower bounds of $R$ in the above cases 
    and
    taking $D\rightarrow \infty $,  we get the required result.
\end{proof}
\begin{example}\label{exm7}
    The warped product $M^{2+m}=\mathbb{R}^2\times_f \mathbb{H}^m$ equipped with the warped product metric $g=g_{\mathbb{R}^n}+\frac{f^2}{2} g_{\mathbb{H}^m}$, where the warping function  $f(x_1, x_2)=\cosh{(x_1+x_2)}$, is of constant scalar curvature $R=-2m(m+1)$. $(M, g)$ is a complete nontrivial warped product soliton with soliton function $h(x_1, x_2, y_1, ..., y_m)=x_1-x_2$.
\end{example}
\begin{corollary}\label{cor6}
    Let $(M^{(n+m)}=B^n\times_f F^m, g=g_B+f^2g_F, h, \rho)$ be a warped product Yamabe gradient soliton with complete base $(B^n, g_B)$ such that the scalar curvature $R$ of $(M^{n+m}, g)$ is bounded above. If  $^{B}Ric(\nabla^B h, \nabla^B h)\ge A {|\nabla^B h|}^2 $ for some $A\in \mathbb{R}$ then  the scalar curvature $R$ is bounded on $M$.
\end{corollary}
\begin{corollary}\label{cor7}
     Under the hypothesis of Theorem \ref{thm5}, the following holds:
    \begin{enumerate}
       \item If $\rho\ge 0$ then $h(x)\ge -(n+m-1)C|A| \frac{r^2}{2}+C_1 r+C_2$.
        \item If $\rho <0$ then  $h(x)\ge \{-(n+m-1)C|A|+\rho\} \frac{r^2}{2}+ C_1r+C_2$.
    \end{enumerate}
    whenever in local coordinates $x=(r, \theta^2,..., \theta^n, \psi^1,...,\psi^m)$.
\end{corollary}

\begin{remark}
    Our assumption is weaker than the hypothesis stated in Theorem 1.4 in \cite{TBAP1}, where the authors assumed the warped product metric to be complete, and for some $K\in \mathbb{R}$
    \begin{equation}\label{FE27}
        ^{B}Ric -m \Hessian \log f-\frac{1}{2(n+m-1)} \Hessian h\ge K.
    \end{equation}
    Whenever $|\nabla^B h|\neq 0$, invoking the unit direction $\frac{\nabla^B h}{|\nabla^B h|}$ in \eqref{FE27} we get
    \begin{eqnarray*}
        \frac{(n+m-1)}{n-1}{^{B}}Ric(\frac{\nabla^B h}{|\nabla^B h|}, \frac{\nabla^B h}{|\nabla^B h|})-\frac{1}{2(n+m-1)}{\Hessian} h(\frac{\nabla^B h}{|\nabla^B h|}, \frac{\nabla^B h}{|\nabla^B h|})+\{\frac{\nabla^B h(f)}{f |\nabla^B h|}\}^2\ge K.
    \end{eqnarray*}
\end{remark}

For an incomplete base, we prove the following result by using a geodesic ball instead of a tubular neighbourhood. The rest of the technique remains the same as in the proof of Theorem \ref{thm5}.
\begin{theorem}\label{thm6}
    Let $(M^{(n+m)}=B^n\times_f F^m, g=g_B+f^2g_F, h, \rho)$ be a nontrivial warped product Yamabe gradient soliton with 
    noncomplete base $(B^n, g_B)$. If $B^n$ is a submanifold of a complete almost Yamabe gradient soliton $(\Bar{B}^n, \Bar{g}_{\Bar{B}}, h, \Bar{\lambda})$ with $h$ having only one critical point $p$ on $\Bar{B}$ and $\Bar{\lambda}|_B=\lambda$ \eqref{FE10} such that  
    \begin{enumerate}
    \renewcommand{\labelenumi}{\roman{enumi}.}
       \item $\Bar{B}^n\backslash \{p\}=B^n$ and $\bar{g}|_B=g_B$;
        \item \label{item:two} there exists $s>0$ such that $R_B-\lambda$ is nonincreasing on $B_s(p)\setminus\{p\}$, where $R_B$ is the scalar curvature of $(B, g_B)$;
        \item $\inf {{^B}Ric(\frac{\nabla^B h}{|\nabla^B h|}}, \frac{\nabla^B h}{|\nabla ^B h|})=A_1$, for some $A_1\in \mathbb{R}$ and $\sup {\Hessian} h(\frac{\nabla^B h}{|\nabla^B h|}, \frac{\nabla^B h}{|\nabla^B h|})=A_2$, for some $A_2\ge 0$ if $\rho \le 0$ or $A_2 \ge -\rho$ if $\rho>0$.
    \end{enumerate}
         Then the scalar curvature $R$ of $(M, g)$ is bounded and there exists a positive constant $C$ such that the following holds:
        \begin{enumerate}
              \item If $\rho\ge 0$ then $-C\{\frac{(n+m-1)^2}{(n-1)}|A_1|+\frac{|A_2|}{2}\} \le R\le A_2+\rho$;
            \item  If $\rho< 0$ then $ -C\{\frac{(n+m-1)^2}{(n-1)}|A_1|+\frac{|A_2|}{2}\}+\rho \le R \le A_2+\rho$.
        \end{enumerate}
        
\end{theorem}

\begin{proof}
    Theorem \ref{thm4} implies that $(\Bar{B}, {\bar{g}}_{\bar{B}})$ is isometric to $([0, \infty), dr^2)\times_{|\nabla^{\bar{B}} h|}(\mathbb{S}^{(n-1)}, \Bar{g}_{can})$. 
   Since, $h$ has a critical point $p\in \Bar{B}$, by fixing $p$ define a function $d:{\Bar{B}}^n\rightarrow \mathbb{R}$ by $d(x):=d(x,p)$.\\ 
   Fix $r_0=2s$.
   For any fixed constant $D>2$, define $v$ on $B$ as in the proof of Theorem \ref{thm5}.\\ 
  If $\min_{x\in B\setminus B_{\frac{r_0}{2D}}(p)} v\ge 0$, then $R\ge 0$ on $B_{\frac{1}{2}Dr_0}(p)\setminus B_{\frac{r_0}{2D}}(p)$.\\
  If $\min_{x\in B\setminus B_{\frac{r_0}{2D}}(p)} v< 0$, then there exists a point $x_1\in B_{Dr_0}(p)\setminus B_{\frac{r_0}{2D}}(p)$ such that $v(x_1)=\alpha R(x_1)<0$. If $x\neq \partial B_\frac{r_0}{2D}$(p), 
   then $\alpha^{'} R(x_1)>0$, $\nabla^B v(x_1)=0$, and $\Delta_B v(x_1)\ge 0$.\\
\textbf{Case I:} When $x_1\in B_{r_0}(p)\setminus B_{\frac{r_0}{2D}}(p)$, Using 
the elliptic equation of the sclar curvature of a Yamabe gradient soliton in \cite[Page 363]{daskalopoulos} and condition $(ii)$  we get for shrinking solitons $x_1\not \in B_{r_0}(p)\setminus B_{\frac{r_0}{2D}}(p) $ and for expanding solitons $R\ge \rho$ on $B_{\frac{Dr_0}{2}}(p)\setminus B_{\frac{r_0}{2D}}(p)$.\\
\textbf{Case II:} Let $x_1\not\in B_{r_0}(p)\setminus B_{\frac{r_0}{2D}}(p)$.
Similar as in the prrof of Theorem \ref{thm5} we get,
\begin{eqnarray*}
\Delta_B v(x_1)=\frac{v(x_1)}{\alpha}[\{\frac{\alpha''}{(Dr_0)^2}-2\frac{{\alpha'}^2}{\alpha}\frac{1}{(Dr_0)^2}\}+\frac{\alpha'}{Dr_0}\{\Delta_B d+g_B(\nabla^B d, \nabla^Blnf^m+\frac{1}{2(n+m-1)}\nabla^B h)\} (x_1)\\
-\frac{v(x_1)}{(n+m-1)}+\frac{\rho}{(n+m-1)}\alpha]
\end{eqnarray*}
Note that the curve $\sigma:[0, d(x_1)] \rightarrow \Bar{B}$ defined by $\sigma(r)=(r, \theta=\theta_1(\text{constant}))$ be a unit speed length minimizing geodesic from $p$ to $x_1$ such that $\sigma^{'}(r)=\phi(r)\nabla^{\bar{B}} k_{\sigma(r)}$, for $r>0$ with $\sigma(0)=p$ and $\sigma^{'}(0)=\frac{\partial}{\partial r}$.
   Let $\{\sigma^{'}(r), E_1(r), E_2(r),..., E_{n-1}(r)\}$ be a parallel orthonormal basis of $T\Bar{B}$ along $\sigma$.
  Let $J_i(r), i=1, 2, ..., (n-1)$, be the Jacobian fields along $\sigma$ with $X_i(0)=0$ and $X_i(d(x_1))=E_i(d(x_1))$. Then, by \cite{schoen} we have,
  \begin{center}
      $\Delta_{\Bar{B}} d(x_1)=\Sigma_{i=1}^{n-1}\int_0^{d(x_1)}[|{{X_i}^{'}}|^2-R(\sigma^{'}, X_i, \sigma^{'}, X_i)] dr$.
  \end{center}
  Define vector fields $Y_i, i=1, 2, 3,..., n-1$, along $\sigma$ as follows:
 \[
    Y_i(r)=
     \begin{cases} 
      \frac{r}{r_0}E_i(r), & \text{if } r\in[0, r_0], \\
      E_i(r), & \text{if }  r\in[r_0, d(x_1)].
    \end{cases}
   \]
 By the standard index comparison theorem,
 \begin{align*}
     \Delta_{\Bar{B}} d(x_1)&\le\Sigma_{i=1}^{n-1}\int_0^{d(x_1)}[|{{Y_i}^{'}}|^2-R(\sigma^{'}, Y_i, \sigma^{'}, Y_i)] dr\\
    \implies \Delta_{B} d(x_1) &\le\int_0^{r_{0}}\{\frac{n-1}{{r_0}^2}-\frac{r^2}{r_0^2}{^{B}Ric(\sigma^{'}, \sigma ^{'})}\} dr-\int_{r_0}^{d(x_1)} {^{B}Ric(\sigma^{'}, \sigma^{'})} dr
 \end{align*}
Since $\inf_B {^B}Ric (\frac{\nabla^B h}{|\nabla^B h|}, \frac{\nabla^B h}{|\nabla^B h|})= A_1$,
\begin{equation}\label{FE12}
    \Delta_B d(x_1)\le\frac{n-1}{r_0}+\frac{2}{3}A_1 r_0-A_1d(x_1)
\end{equation}


Let $p_0$ be the point on $B$ such that $d(p, p_0)=r_0$ and lies on the shortest geodesic joining $p$ and $x_1$. Since $^{B}Ric(\nabla^B h, \nabla^B h)\ge A_1 {|\nabla^B h|}^2$ and $^{B}{Hess}(h)(\nabla^B h, \nabla^B h)\le A_2{|\nabla^B h|^2}$, by a similar procedure to the proof of Theorem \ref{thm5}, we deduce that, 
\begin{align*}
     g_B(\nabla_B d, \nabla^Blnf^m+\frac{1}{2(n+m-1)}\nabla^B h)(x_1)\le -\frac{m A_1}{n-1} d(x_1)+\frac{A_2}{2(n+m-1)} d(x_1)\\
     +|\nabla^B h|_{p_0}+\sigma' (\ln f)_{p_0}
\end{align*}
Hence,
\begin{eqnarray*}
        \Delta_B v(x_1)\le \frac{|v(x_1)|}{\alpha}[\{\frac{|\alpha''|}{(Dr_0)^2}+2\frac{|\alpha'|}{\alpha}\frac{1}{(Dr_0)^2}\}+\frac{|\alpha'|}{Dr_0}\{\frac{n-1}{r_0}+\frac{2|A_1|r_0}{3}+\frac{(n+m-1)|A_1|}{(n-1)} d(x_1)\\+\frac{|A_2|}{2(n+m-1)}d(x_1)+|\nabla^B h|_{p_0}+|\sigma'(\ln f)_{p_0}|\}
       -\frac{|v(x_1)|}{(n+m-1)}-\frac{\rho}{(n+m-1)}\alpha]
    \end{eqnarray*}
    Let $C_1=|\nabla^B h|_{p_0}$ and $C_2=|\sigma'(\ln f)|_{p_0}$. Following the same method to the proof of Theorem \ref{thm5}, 
    We get the required result.
\end{proof}
\begin{example}\cite{topping}\label{exm5}
    Let $\mathbb{R}_{+}=\{r\in\mathbb{R}: r>0\}$. The punctured Hamilton cigar soliton 
    $\mathbb{C}^2_*$ is the warped product $\mathbb{R}_{+}\times_{\tanh{r}} \mathbb{S}^1$ equipped with the warped metric $g=dr^2+\tanh^2{r} g_{can}$ of scalar curvature $R=\frac{4}{\cosh^2{r}}$. $(\mathbb{C}^2_*, g)$ is a noncomplete nontrivial steady warped product Yamabe gradient soliton of bounded scalar curvature with the soliton function
        $h(r, \theta)= 2\ln \cosh{r}$.
\end{example}
\begin{example}
     Consider the punctured Euclidean $n-$ space $\mathbb{E}^n_*$  equipped with the metric $g_{\mathcal{E}}=dr^2+r^2 g_{can}$, where $g_{can}$ is the round metric on the sphere $\mathbb{S}^{n-1}$. Then $M^{2n}=\mathbb{E}^n_*\times_r \mathbb{S}^n$ equipped with the metric $g_M=g_\mathcal{E}+r^2\Bar{g}$, where $\Bar{g}=\frac{1}{3}g_{can}$is a noncomplete nontrivial warped product  Yamabe gradient soliton of zero scalar curvature with the soliton function $h(r, \theta, \phi)=-\frac{\rho r^2}{2}$, where $\rho$ is a nonzero constant. 
\end{example}
\begin{corollary}\label{cor9}
     Under the hypothesis of Theorem \ref{thm6}, the following holds:
    \begin{enumerate}
       \item If $\rho\ge 0$ then $h(x)\ge -C\{\frac{(n+d-1)^2}{(n-1)}|A_1|+\frac{|A_2|}{2}\}\frac{d(x)^2}{2}+C_1 d(x)+C_2$.
        \item If $\rho <0$ then  $h(x)\ge \{-C\{\frac{(n+d-1)^2}{(n-1)}|A_1|+\frac{|A_2|}{2}\}+\rho\} \frac{d(x)^2}{2}+ C_1d(x)+C_2$.
        
    \end{enumerate}
    Where $d$ is the distance function from the critical point  of $h$ on $\Bar{B}$ and $C_1, C_2\in\mathbb{R}$.
\end{corollary}

\printbibliography
\vspace{2mm}
\address{INDIAN INSTITUTE OF SCIENCE EDUCATION AND RESEARCH BHOPAL, MADHYA PRADESH-462066, INDIA.}

\email{\textit{E-mail address: jahnabi22@iiserb.ac.in}}

\email{\textit{E-mail address: anand@iiserb.ac.in}}

\end{document}